\documentclass[11pt]{amsart}
\input{epsf}
\usepackage{amssymb,latexsym}
\topskip  \newtheorem{theorem}{Theorem}
\newtheorem{proposition}[theorem]{Proposition}
\newtheorem{corollary}[theorem]{Corollary}

\newtheorem{remark}{Remark}
\newtheorem{lemma}[theorem]{Lemma}

\begin{document}
\title{Congruence successions in compositions}

%

%

\author{Toufik Mansour}
\address{Department of Mathematics, University of Haifa, 31905 Haifa, Israel}
\email{tmansour@univ.haifa.ac.il}

\author{Mark Shattuck}
\address{Department of Mathematics, University of Tennessee, Knoxville, TN 37996 USA}
\email{shattuck@math.utk.edu}

\author{Mark C. Wilson}
\address{Department of Computer Science, University of Auckland, Private Bag 92019 Auckland, New Zealand}
\email{mcw@cs.auckland.ac.nz}


\maketitle
\begin{abstract}
A \emph{composition} is a sequence of positive integers, called \emph{parts}, having a fixed sum.  By an \emph{$m$-congruence succession}, we will mean a pair of adjacent parts $x$ and $y$ within a composition such that $x\equiv y~(\text{mod}~m)$.  Here, we consider the problem of counting the compositions of size $n$ according to the number of $m$-congruence successions, extending recent results concerning successions on subsets and permutations.  A general formula is obtained, which reduces in the limiting case to the known generating function formula for the number of Carlitz compositions.  Special attention is paid to the case $m=2$, where further enumerative results may be obtained by means of combinatorial arguments.  Finally, an asymptotic estimate is provided for the number of compositions of size $n$ having no $m$-congruence successions.

\end{abstract}

\section{Introduction}

Let $n$ be a positive integer.  A \emph{composition} $\sigma=\sigma_1\sigma_2\cdots\sigma_d$ of $n$ is any sequence of positive integers whose sum is $n$.  Each summand $\sigma_i$ is called a \emph{part} of the composition.  If $n,d \geq 1$, then let $\mathcal{C}_{n,d}$ denote the set of compositions of $n$ having exactly $d$ parts and $\mathcal{C}_n=\cup_{d=1}^n\mathcal{C}_{n,d}$.  By convention, there is a single composition of $n=0$ having zero parts.

If $m \geq 1$ and $0 \leq r \leq m-1$, by an {\em $(m,r)$-congruence succession} within a composition $\sigma=\sigma_1\sigma_2\cdots\sigma_d$, we will mean an index $i$ for which $\pi_{i+1}\equiv\pi_i+r~(\text{mod}~m)$.  An $(m,r)$-congruence succession in which $r=0$ will be referred to as an {\em $m$-congruence succession}, the $m=2$ case being termed a {\em parity succession}.  A  \emph{parity-alternating} composition is one that contains no parity successions, that is, the parts alternate between even and odd values.  This concept of parity succession for compositions extends an earlier one that was introduced for subsets \cite{M2} and later considered on permutations \cite{M1}.  The terminology is an adaptation of an analogous usage in the study of integer sequences $(p_1,p_2,\ldots)$ in which a succession refers to a pair $p_i,p_{i+1}$ with $p_{i+1}=p_i+1$ (see, e.g., \cite{AM,Tan1,KMW}).  For other related problems involving restrictions on compositions, the reader is referred to the text \cite{HM} and such papers as \cite{CH,HK}.

Enumerating finite discrete structures according to the parity of individual elements perhaps started with the following formula of Tanny \cite{Tan2} for the number $g(n,k)$ of alternating $k$-subsets of $[n]$ given by

$$g(n,k)=\binom{\lfloor\frac{n+k}{2}\rfloor}{k}+\binom{\lfloor\frac{n+k-1}{2}\rfloor}{k}, \qquad 1 \leq k \leq n.$$

This result was recently generalized to any modulus in \cite{MM} and in terms of counting by successions in \cite{M2}.  Tanimoto \cite{Tani} considered a comparable version of the problem on permutations in his investigation of signed Eulerian numbers. There one finds the formula for the number $h(n)$ of parity-alternating permutations of length $n$ given by
$$h(n)=\frac{3+(-1)^n}{2}\left\lfloor\frac{n}{2}\right\rfloor!\left\lfloor\frac{n+1}{2}\right\rfloor!,$$
which has been generalized in terms of succession counting in \cite{M1}; see also \cite{M3}.

In the next section, we consider the problem of counting compositions of $n$ according to the number of $(m,r)$-congruence successions, as defined above, and derive an explicit formula for the generating function for all $m$ and $r$ (see Theorem \ref{t1} below).  When $r=0$, we obtain as a corollary a relatively simple expression for the generating function $F_m$ which counts compositions according to the number of $m$-congruence successions.  Letting $m\rightarrow \infty$ and taking the variable in $F_m$ which marks the number of $m$-congruence successions to be zero recovers the generating function for the number of \emph{Carlitz} compositions, i.e., those having no consecutive parts equal; see, e.g., \cite{GoHi2002}.

In the third section, we obtain some enumerative results concerning the case $m=2$.  In particular, we provide a bijective proof for a related recurrence and enumerate, in two different ways, the parity-alternating compositions of size $n$.  As a consequence, we obtain a combinatorial proof of a pair of binomial identities which we were unable to find in the literature.  In the final section, we provide asymptotic estimates for the number of compositions of size $n$ having no $m$-congruence successions as $n \rightarrow \infty$, which may be extended to compositions having any fixed number of successions.

\section{Counting compositions by number of $(m,r)$-congruence successions}

We will say that the sequence $\pi=\pi_1\pi_2\cdots\pi_d$ has an {\em $(m,r)$-congruence succession} at index $i$ if $\pi_{i+1}\equiv\pi_i+r~(\text{mod}~m)$, where $0 \leq r \leq m-1$. We will denote the number of $(m,r)$-congruence successions within a sequence $\pi$ by $cl_{m,r}(\pi)$. Let $R_{m,r;a}(x,y,q)=R_{a}(x,y,q)$ be the generating function for the number of compositions of $n$ with exactly $d$ parts whose first part is $a$ according to the statistic $cl_{m,r}$, that is,
$$R_a(x,y,q)=\sum_{n\geq0}\sum_{d=0}^n x^ny^d \left(\sum_{\pi=a\pi'\in\mathcal{C}_{n,d}}q^{cl_{m,r}(\pi)}\right).$$
Clearly, we have $R_{m+a}(x,y,q)=x^{m}R_a(x,y,q)$ for all $a\geq1$. Let $R_{m,r}(x,y,q)=R(x,y,q)=1+\sum_{a\geq1}R_a(x,y,q)$. By the definitions, we have
$$R_a(x,y,q)=x^ayR(x,y,q)+x^ay(q-1)\sum R_t(x,y,q),$$
for all $a\geq 1$, where the sum is taken over all positive integers $t$ such that
$t\equiv\ a+r~(\text{mod}~m)$.
Hence,
$$\sum_{i\geq0}R_{im+a}(x,y,q)=\frac{x^ay}{1-x^m}R(x,y,q)+\frac{x^ay(q-1)}{1-x^m}\sum_{i\geq0}R_{im+a+r}(x,y,q),$$
if $1 \leq a \leq m-r$, and
$$\sum_{i\geq0}R_{im+a}(x,y,q)=\frac{x^ay}{1-x^m}R(x,y,q)+\frac{x^ay(q-1)}{1-x^m}\sum_{i\geq0}R_{im+a+r-m}(x,y,q),$$
if $m-r+1 \leq a \leq m$.   The last two equalities may be expressed as
\begin{align}
G_{j}(x,y,q)&=\frac{x^jy}{1-x^m}R(x,y,q)+\frac{x^jy(q-1)}{1-x^m}G_{j+r}(x,y,q),\quad 1\leq j\leq m-r,\notag\\
G_{j}(x,y,q)&=\frac{x^jy}{1-x^m}R(x,y,q)+\frac{x^jy(q-1)}{1-x^m}G_{j+r-m}(x,y,q),\quad m-r+1\leq j\leq m,\label{eqr1}
\end{align}
where $G_j(x,y,q)=\sum_{i\geq0}R_{im+j}(x,y,q)$.

In order to find an explicit formula for $G_j(x,y,q)$, we will need the following lemma.
\begin{lemma}\label{lem01}
Suppose $x_j=a_j+b_jx_{j+r}$ for all $j=1,2,\ldots,m-r$ and $x_j=a_j+b_jx_{j+r-m}$ for all $j=m-r+1,m-r+2,\ldots,m$. Let $s=\gcd(m,r)$ and $p=m/s$. Then for all $j=1,2,\ldots,s$ and $\ell=0,1,\ldots,p-1$, we have
$$x_{j+\ell r}=\sum_{i=\ell}^{\ell+p-1} \frac{a_{j+ir}\prod_{k=\ell}^{i-1}b_{j+kr}}{1-\prod_{k=\ell}^{\ell+p-1}b_{j+kr}},$$
where $x_{j+m}=x_j$, $a_{j+m}=a_j$ and $b_{j+m}=b_j$.
\end{lemma}
\begin{proof}
Let $j=1,2,\ldots,s$. By definition of the sequence $x_j$ and $m$-periodicity, we may write
\begin{align*}
x_j&=a_j+b_jx_{j+r}=a_j+b_ja_{j+r}+b_jb_{j+r}x_{j+2r}\\
&=\cdots=a_j+b_ja_{j+r}+\cdots+b_jb_{j+r}\cdots b_{j+(p-2)r}a_{j+(p-1)r}+b_jb_{j+r}\cdots b_{j+(p-1)r}x_{j+pr}.
\end{align*}
Since $pr\equiv0~(\text{mod}~m)$, we have
$$x_j=\sum_{i=0}^{p-1} \frac{a_{j+ir}\prod_{k=0}^{i-1}b_{j+kr}}{1-\prod_{k=0}^{p-1}b_{j+kr}}.$$
More generally, for any $\ell=0,1,\ldots,p-1$,
$$x_{j+\ell r}=\sum_{i=\ell}^{\ell+p-1} \frac{a_{j+ir}\prod_{k=\ell}^{i-1}b_{j+kr}}{1-\prod_{k=\ell}^{\ell+p-1}b_{j+kr}}.$$
\end{proof}

Let us denote by $\overline{t}$ the member of $\{1,2,\ldots,m\}$ such that $t\equiv \overline{t}~(\text{mod}~m)$ for a positive integer $t$.  When $a_j=\frac{x^jy}{1-x^m}R(x,y,q)$ and $b_j=\frac{x^jy(q-1)}{1-x^m}$ for $1 \leq j \leq m$ in Lemma \ref{lem01}, we get
\begin{align}
x_{j+\ell r} &=\sum_{i=\ell}^{\ell+p-1}\frac{\frac{x^{\overline{j+ir}}y}{1-x^m}R(x,y,q)\prod_{k=\ell}^{i-1}\frac{x^{\overline{j+kr}}y(q-1)}{1-x^m}}{1-\prod_{k=\ell}^{\ell+p-1}\frac{x^{\overline{j+kr}}y(q-1)}{1-x^m}}\notag\\
&=\frac{R(x,y,q)}{1-\left(\frac{y(q-1)}{1-x^m}\right)^{p}\prod_{k=\ell}^{\ell+p-1}x^{\overline{j+kr}}}\sum_{i=0}^{p-1}\frac{x^{\overline{j+(i+\ell)r}}y^{i+1}(q-1)^i\prod_{k=\ell}^{i+\ell-1}x^{\overline{j+kr}}}{(1-x^m)^{i+1}}\label{eqr2},
\end{align}
for all $j=1,2,\ldots,s$ and $\ell=0,1,\ldots,p-1$, where $s=\gcd(m,r)$ and $p=m/s$. By \eqref{eqr1}, we have $G_{\overline{j+\ell r}}(x,y,q)=x_{j+\ell r}=x_{\overline{j+\ell r}}$, where $x_{j+\ell r}$ is given by \eqref{eqr2}.  Note that the set of indices $j+\ell r$ for $1\leq j \leq s $ and $0\leq \ell \leq p-1$ is a complete residue set $(\text{mod}~m)$. Using \eqref{eqr2} and the fact that $R(x,y,q)=1+\sum_{a=1}^m G_a(x,y,q)$, we obtain the following result.

\begin{theorem}\label{t1}
If $m \ge1 $, $0\leq r\leq m-1$, $s=\gcd(m,r)$ and $p=m/s$, then
\begin{equation}\label{t1e1}
R_{m,r}(x,y,q)=\frac{1}{1-\sum\limits_{j=1}^s\sum\limits_{\ell=0}^{p-1}\sum\limits_{i=0}^{p-1}\frac{x^{\overline{j+(i+\ell)r}}y^{i+1}(q-1)^i\prod_{k=\ell}^{i+\ell-1}x^{\overline{j+kr}}}{(1-x^m)^{i+1}\left(1-\left(\frac{y(q-1)}{1-x^m}\right)^{p}\prod_{k=\ell}^{\ell+p-1}x^{\overline{j+kr}}\right)}}.
\end{equation}
\end{theorem}

Note that in general we are unable to simplify the number theoretic product $\prod_{k=\ell}^{i+\ell-1}x^{\overline{j+kr}}$ appearing in \eqref{t1e1}.

Let us say that the sequence $\pi=\pi_1\pi_2\cdots\pi_d$ has an {\em $m$-congruence succession} at index $i$ if $\pi_{i+1}\equiv\pi_i~(\text{mod}~m)$ and denote the number of $m$-congruence successions in a sequence $\pi$ by $cl_m(\pi)$.  Let
$$F_m(x,y,q)=\sum_{n\geq0}\sum_{d=0}^n x^ny^d \left(\sum_{\pi\in\mathcal{C}_{n,d}}q^{cl_{m}(\pi)}\right).$$
Taking $r=0$ in \eqref{t1e1}, and noting $s=\gcd(m,0)=m$, gives the following result.

\begin{corollary}\label{c1}
If $m \geq 1$, then
\begin{equation}\label{c1e1}
F_m(x,y,q)=\frac{1}{1-\sum_{a=1}^m\left(\frac{x^ay}{1-x^m-x^ay(q-1)}\right)}.
\end{equation}
\end{corollary}

Letting $q=0$ and $m \rightarrow \infty$ in \eqref{c1e1} yields the generating function for the number of compositions having no $m$-congruence successions for all large $m$. Note that the only possible such compositions are those having no two adjacent parts the same. Thus, we get the following formula for the generating function which counts the Carlitz compositions according to the number of parts.

\begin{corollary}\label{c1a}
We have
\begin{equation}\label{c1ae1}
F_\infty(x,y,0)=\frac{1}{1-\sum_{a=1}^{\infty}\frac{x^ay}{1+x^ay}}.
\end{equation}
\end{corollary}

Let us close this section with a few remarks.

\begin{remark}
Letting $q=1$ in \eqref{t1e1} gives
$$R_{m,r}(x,y,1)=\frac{1}{1-\frac{y}{1-x^m}\sum_{j=1}^s\sum_{\ell=0}^{p-1}x^{\overline{j+\ell r}}}=\frac{1}{1-\frac{y}{1-x^m}\sum_{a=1}^mx^a}=\frac{1-x}{1-x-xy},$$
which agrees with the generating function for the number of compositions of $n$ having $d$ parts.
\end{remark}

\begin{remark}
In \cite{C}, the generating function for the number $c(n,d)$ of Carlitz compositions of $n$ having $d$ parts was obtained as
\begin{equation}\label{reme1}
\sum_{n\geq0}\sum_{d=0}^n c(n,d)x^ny^d=\frac{1}{1+\sum_{j \geq 1}\frac{(-xy)^j}{1-x^j}}.
\end{equation}
Note that formulas \eqref{c1ae1} and \eqref{reme1} are seen to be equivalent since
$$\sum_{a\geq1}\frac{x^ay}{1+x^ay}=\sum_{a\geq1}\sum_{j \geq 1}(-1)^{j-1}x^{aj}y^j=\sum_{j\geq1}(-1)^{j-1}y^j\sum_{a \geq 1}x^{ja}=\sum_{j\geq1}(-1)^{j-1}\frac{(xy)^j}{1-x^j}.$$
\end{remark}

\begin{remark}
Letting $m=1$ in \eqref{c1e1} gives
$$F_1(x,y,q)=\frac{1-x-xy(q-1)}{1-x-xyq}.$$
This formula may also be realized directly upon noting in this case that $q$ marks the number of parts minus one in any non-empty composition, whence
$$F_1(x,y,q)=1+\frac{1}{q}\left( \frac{1-x}{1-x-xyq}-1\right).$$
\end{remark}

\section{Combinatorial results}

We will refer to an $m$-congruence succession when $m=2$ as a \emph{parity succession}, or just a \emph{succession}.  In this section, we will provide some combinatorial results concerning successions in compositions.  Let $F(x,y,q)=F_2(x,y,q)$ denote the generating function which counts the compositions of $n$ having $d$ parts according to the number of parity successions.  Taking $m=2$ in Corollary \ref{c1} gives
\begin{equation}\label{coe1}
F(x,y,q)=\frac{(1-x^2-xy(q-1))(1-x^2-x^2y(q-1))}{(1-x^2)^2-x^3y^2-xy(1-x^2)(1+x)q+x^3y^2q^2}.
\end{equation}
Let $\mathcal{C}_{n,d,a}$ denote the subset of $\mathcal{C}_{n,d}$ whose members contain exactly $a$ successions and let $c(n,d,a)=|\mathcal{C}_{n,d,a}|$.  Comparing coefficients of $x^ny^dq^a$ on both sides of \eqref{coe1} yields the following recurrence satisfied by the array $c(n,d,a)$.

\begin{theorem}\label{t2}
If $n \geq 4$ and $d \geq 3$, then
\begin{align}
c(n,d,&a)=c(n-1,d-1,a-1)+2c(n-2,d,a)+c(n-2,d-1,a-1)+c(n-3,d-2,a)\notag\\
&-c(n-3,d-1,a-1)-c(n-3,d-2,a-2)-c(n-4,d,a)-c(n-4,d-1,a-1).\label{theq1}
\end{align}
\end{theorem}

We can also provide a combinatorial proof of \eqref{theq1}, rewritten in the form
\begin{align}
(c(n,d,a)&-c(n-2,d,a))+c(n-3,d-2,a-2)=\notag\\
&(c(n-1,d-1,a-1)-c(n-3,d-1,a-1))+(c(n-2,d,a)-c(n-4,d,a))\notag\\
&+(c(n-2,d-1,a-1)-c(n-4,d-1,a-1))+c(n-3,d-2,a).\label{theq2}
\end{align}
To do so, let $\mathcal{B}_{n,d,a}$ denote the subset of $\mathcal{C}_{n,d,a}$ all of whose members end in a part of size $1$ or $2$.  Note that for all $n$, $d$, and $a$, we have
$$|\mathcal{B}_{n,d,a}|=c(n,d,a)-c(n-2,d,a),$$
by subtraction, since $c(n-2,d,a)$ counts each member of $\mathcal{C}_{n,d,a}$ whose last part is of size $3$ or more (to see this, add two to the last part of any member of $\mathcal{C}_{n-2,d,a}$, which leaves the number of parts and successions unchanged).

So to show \eqref{theq2}, we define a bijection between the sets
$$\mathcal{B}_{n,d,a}\cup \mathcal{C}_{n-3,d-2,a-2}\text{ and } \mathcal{B}_{n-1,d-1,a-1}\cup \mathcal{B}_{n-2,d,a}\cup \mathcal{B}_{n-2,d-1,a-1}\cup \mathcal{C}_{n-3,d-2,a}.$$

For this, we refine the sets as follows.  In the subsequent definitions, $x$, $y$, and $z$ will denote an odd number, an even number, or a number greater than or equal three, respectively.  First, let $\mathcal{B}_{n,d,a}^{(i)}$, $1 \leq i \leq 4$, denote, respectively, the subsets of $\mathcal{B}_{n,d,a}$ whose members (1) end in either $1+1$ or $x+1+2$ for some $x$, (2) end in $y+2+1$ or $2+2$ for some $y$, (3) end in $x+2+1$ or $y+1+2$, or (4) end in $z+1$ or $z+2$ for some $z$.  Let $\mathcal{B}_{n-1,d-1,a-1}^{(i)}$, $1 \leq i \leq 3$, denote the subsets of $\mathcal{B}_{n-1,d-1,a-1}$ whose members end in $1$, $x+2$ for some $x$, or $y+2$ for some $y$, respectively.  Finally, let $\mathcal{B}_{n-2,d-1,a-1}^{(i)}$, $1 \leq i \leq 3$, denote the subsets of $\mathcal{B}_{n-2,d-1,a-1}$ whose members end in $x+1$, $y+1$, or $2$, respectively.

So we seek a bijection between
$\left(\cup_{i=1}^4 \mathcal{B}_{n,d,a}^{(i)}\right)\cup \mathcal{C}_{n-3,d-2,a-2}$ and $$\left(\cup_{i=1}^3 \mathcal{B}_{n-1,d-1,a-1}^{(i)}\right)\cup\left(\cup_{i=1}^3 \mathcal{B}_{n-2,d-1,a-1}^{(i)}\right)\cup \mathcal{B}_{n-2,d,a}\cup \mathcal{C}_{n-3,d-2,a}.$$

Simple correspondences as described below show the following:
\begin{align*}
(i)& \quad|\mathcal{B}_{n,d,a}^{(1)}|=|\mathcal{B}_{n-1,d-1,a-1}^{(1)}\cup \mathcal{B}_{n-1,d-1,a-1}^{(2)}|,\\
(ii)& \quad |\mathcal{B}_{n,d,a}^{(2)}|=|\mathcal{B}_{n-2,d-1,a-1}^{(2)}\cup \mathcal{B}_{n-2,d-1,a-1}^{(3)}|,\\
(iii)& \quad |\mathcal{B}_{n,d,a}^{(3)}|=|\mathcal{C}_{n-3,d-2,a}|,\\
(iv)& \quad |\mathcal{B}_{n,d,a}^{(4)}|=|\mathcal{B}_{n-2,d,a}|,\\
(v)& \quad |\mathcal{C}_{n-3,d-2,a-2}|=|\mathcal{B}_{n-2,d-1,a-1}^{(1)}\cup \mathcal{B}_{n-1,d-1,a-1}^{(3)}|.
\end{align*}

For (i), we remove the right-most $1$ within a member of $\mathcal{B}_{n,d,a}^{(1)}$, while for (ii), we remove the right-most $2$ within a member of
$\mathcal{B}_{n,d,a}^{(2)}$.  To show (iii), we remove the final two parts of $\lambda \in \mathcal{B}_{n,d,a}^{(3)}$ to obtain the composition $\lambda'$.  Note that $\lambda' \in \mathcal{C}_{n-3,d-2,a}$ and that the mapping $\lambda \mapsto \lambda'$ is reversed by adding $1+2$ or $2+1$ to a member of $\mathcal{C}_{n-3,d-2,a}$, depending on whether the last part is even or odd, respectively.  For (iv), we subtract two from the penultimate part of $\lambda \in \mathcal{B}_{n,d,a}^{(4)}$, which leaves the number of successions unchanged.  Finally, for (v), we add either a part of size $1$ or $2$ to $\lambda \in \mathcal{C}_{n-3,d-1,a-2}$, depending on whether the last part of $\lambda$ is odd or even, respectively.  Combining the correspondences used to show (i)--(v) yields the desired bijection and completes the proof.  \hfill \qed

We will refer to a composition having no parity successions as \emph{parity-alternating}.  We now wish to enumerate parity-alternating compositions having a fixed number of parts.  Setting $q=0$ in \eqref{coe1}, and expanding, gives
\begin{align*}
F(&x,y,0)=\frac{(1-x^2+x^2y)(1-x^2+xy)}{(1-x^2)^2-x^3y^2}
=\frac{\left(1+\frac{x^2y}{1-x^2}\right)\left(1+\frac{xy}{1-x^2}\right)}{1-\frac{x^3y^2}{(1-x^2)^2}}\\
&=\left(1+\frac{x^2y}{1-x^2}\right)\left(1+\frac{xy}{1-x^2}\right)\sum_{i\geq0}\frac{x^{3i}y^{2i}}{(1-x^2)^{2i}}\\
&=\sum_{i \geq 0}\left(2y^{2i}\sum_{j \geq 2i-1}\binom{j}{2i-1}x^{2j-i+2}+y^{2i+1}\sum_{j \geq 2i}\binom{j}{2i}x^{2j-i+2}
+y^{2i+1}\sum_{j \geq 2i}\binom{j}{2i}x^{2j-i+1}\right).
\end{align*}
Extracting the coefficient of $x^ny^m$ in the last expression yields the following result.
\begin{proposition}\label{p1}
If $n\geq 1$ and $d \geq0$, then
\begin{equation}\label{p1e1}
c(n,2d,0)=\left\{\begin{array}{ll}2\binom{\frac{n+d}{2}-1}{2d-1},&\mbox{if } n \equiv d ~(\mbox{mod}~2);\\\\
0,&\mbox{otherwise},
\end{array}\right.
\end{equation}
and
\begin{equation}\label{p1e2}
c(n,2d+1,0)=\left\{\begin{array}{ll}\binom{\frac{n+d}{2}-1}{2d},&\mbox{if } n \equiv d ~(\mbox{mod}~2);\\\\
\binom{\frac{n+d-1}{2}}{2d},&\mbox{otherwise}.
\end{array}\right.
\end{equation}
\end{proposition}

It is instructive to give combinatorial proofs of \eqref{p1e1} and \eqref{p1e2}.  For the first formula, suppose $\lambda \in \mathcal{C}_{n,2d,0}$. Then $n$ and $d$ must have the same parity since the parts of $\lambda$ alternate between even and odd values.  In this case, the number of possible $\lambda$ is twice the number of integral solutions to the equation
\begin{equation}\label{p1e3}
\sum_{i=1}^d(x_i+y_i)=n,
\end{equation}
where each $x_i$ is even, each $y_i$ is odd, and $x_i,y_i>0$.
Note that the number of solutions to \eqref{p1e3} is the same as the number of positive integral solutions to
$\sum_{i=1}^d(u_i+v_i)=\frac{n+d}{2},$ which is $\binom{\frac{n+d}{2}-1}{2d-1}$, upon letting $u_i=\frac{x_i}{2}$ and $v_i=\frac{y_i+1}{2}$.  Thus, there are $2\binom{\frac{n+d}{2}-1}{2d-1}$ members of $\mathcal{C}_{n,2d,0}$ when $n$ and $d$ have the same parity, which gives \eqref{p1e1}.

On the other hand, note that members of $\mathcal{C}_{n,2d+1,0}$, where $n$ and $d$ are of the same parity, are synonymous with positive integral solutions to
\begin{equation}\label{p1e4}
\sum_{i=1}^d(x_i+y_i)+z=n,
\end{equation}
where the $x_i$ are even, the $y_i$ are odd, and $z$ is even.  Upon adding $1$ to each $y_i$, and halving, the number of such solutions is seen to be $\binom{\frac{n+d}{2}-1}{2d}$.  Similarly, there are $\binom{\frac{n+d-1}{2}}{2d}$ members of $\mathcal{C}_{n,2d+1,0}$ when $n$ and $d$ differ in parity, which gives \eqref{p1e2}. \hfill \qed\\

Let $a(n)=\sum_{d=0}^n c(n,d,0)$.  Note that $a(n)$ counts all parity-alternating compositions of length $n$. Taking $y=1$ and $q=0$ in \eqref{coe1} gives
$$\sum_{n\geq0}a(n)x^n=F(x,1,0)=\frac{1+x-x^2}{(1-x^2)^2-x^3},$$
and extracting the coefficient of $x^n$ yields the following result.

\begin{proposition}\label{p2}
If $n\geq 4$, then
\begin{equation}\label{p2e1}
a(n)=2a(n-2)+a(n-3)-a(n-4),
\end{equation}
with $a(0)=a(1)=a(2)=1$ and $a(3)=3$.
\end{proposition}

We provide a combinatorial argument for recurrence \eqref{p2e1}, the initial values being clear.  Let us first define two classes of compositions.  By a \emph{type A (colored) composition} of $n$, we will mean one whose parts are all odd numbers greater than $1$ in which a part of size $i$ is assigned one of $\frac{i-1}{2}$ possible colors. A composition $\rho$ of $n$ is of \emph{type B} if it is of the form $\rho=a+\lambda$, where $a$ is a part of any size (not colored) and $\lambda$ is a composition of $n-a$ of type A.  Given $n \geq 1$, let $\mathcal{S}_n$ denote the (multi-) set consisting of two copies of each composition of $n$ of type A, let $\mathcal{T}_n$ denote the set consisting of all compositions of $n$ of type B, and let $\mathcal{R}_n=\mathcal{S}_n\cup \mathcal{T}_n$.  (For convenience, we take $R_0$ to consist of the empty composition in a set by itself.)

If $k_i$ denotes the number of the color assigned to some part of size $i$ within a composition of type A, then replacing each part $i$ with either $(i-2k_i)+2k_i$ or $2k_i+(i-2k_i)$ yields all members of $\mathcal{C}_n$ having an even number of parts and no parity successions.  Doing the same for every part but the first within a composition of type $B$ (note in this case, the decomposition used for each $i$ is determined by the parity of the first part $a$) yields all members of $\mathcal{C}_n$ having an odd number of parts and no parity successions.  It follows that $|\mathcal{R}_n|=a(n)$.

It then remains to show that $|\mathcal{R}_n|$ satisfies the recurrence \eqref{p2e1}.  By a \emph{maximal} (colored) part of size $i$ within a member of $\mathcal{R}_n$, we will mean one which has been assigned the color $\frac{i-1}{2}$ (note that any part of size $3$ is maximal).  Let $\mathcal{R}_n'$ denote the subset of $\mathcal{R}_n$ consisting of all type $A$ members whose first part is maximal together with all type $B$ members whose first part is $1$ or $2$.  Upon increasing the length of the first part within a member of $\mathcal{R}_{n-2}$ by two (keeping the color the same, if that member belongs to $\mathcal{S}_{n-2}$), one sees that $|\mathcal{R}_n'|=a(n)-a(n-2)$, by subtraction.  To complete the proof of \eqref{p2e1}, we then define a bijection between the sets $\mathcal{R}_{n-2}\cup\mathcal{R}_{n-3}$ and $\mathcal{R}_n'\cup \mathcal{R}_{n-4}$, where $n \geq 4$.

We may assume $n \geq 5$, for the equivalence of the sets in question is clear if $n=4$. Let $\mathcal{S}_n'$ and $\mathcal{T}_n'$ denote the subsets of $\mathcal{R}_n'$ consisting of its type $A$ and type $B$ members, respectively.  To complete the proof, it suffices to define bijections between the sets $\mathcal{S}_{n-2}\cup\mathcal{S}_{n-3}$ and $\mathcal{S}_n'\cup \mathcal{S}_{n-4}$ and between the sets $\mathcal{T}_{n-2}\cup\mathcal{T}_{n-3}$ and $\mathcal{T}_n'\cup \mathcal{T}_{n-4}$.

For the first bijection, if $\lambda \in \mathcal{S}_{n-2}$, then we either increase or decrease the length of the first part of $\lambda$ by two, depending on whether or not this part is maximal (if so, we also increase the color assigned to the part by one, and if not, the color is kept the same).  Note that this yields all members of $S_n'$ whose first part is at least five as well as all members of $\mathcal{S}_{n-4}$.  If $\lambda \in \mathcal{S}_{n-3}$, then we append a colored part of size three to the beginning of $\lambda$, which yields all members of $S_n'$ starting with three.

For the second bijection, we consider cases concerning $\lambda \in \mathcal{T}_{n-2} \cup \mathcal{T}_{n-3}$.  If $\lambda \in \mathcal{T}_{n-2}$ starts with $1$, then we increase the second part of $\lambda$ by two (keeping the assigned color the same) to obtain $\lambda^* \in \mathcal{T}_n'$ starting with $1$ where the second part is not maximal.  If $\lambda \in \mathcal{T}_{n-3}$ starts with $1$, then we replace this $1$ with $2$ and increase the second part of $\lambda$ by two (again, keeping the assigned color the same) to obtain $\lambda^* \in \mathcal{T}_n'$.  Combining the previous two cases then yields all members of $\mathcal{T}_n'$ whose second part is not maximal.  If $\lambda \in \mathcal{T}_{n-2}$ starts with a part $i$ of size two or more, then we append a $2$ to the beginning of $\lambda$ if $i$ is odd and we append a $1$ to the beginning of $\lambda$ and replace $i$ with $i+1$ if $i$ is even. In either case, we take the second part to be maximal in the resulting composition $\lambda^*$ belonging to $\mathcal{T}_n'$.  Finally, if $\lambda \in \mathcal{T}_{n-3}$ starts with a part of size two or more, then we subtract one from this part to obtain $\lambda^* \in \mathcal{T}_{n-4}$.  It may be verified that the composite mapping $\lambda \mapsto \lambda^*$ yields the desired bijection. \hfill \qed

Summing the formulas in Proposition \ref{p1} over $d$ with $n$ fixed, using the fact $\binom{a}{b}=\binom{a-2}{b}+2\binom{a-2}{b-1}+\binom{a-2}{b-2}$, and equating with the result in Proposition \ref{p2} yields a combinatorial proof of the following pair of binomial identities, which we were unable to find in the literature.

\begin{corollary}\label{c2}
If $n\geq0$, then
\begin{equation}\label{c2e1}
a(2n)=\sum_{d=0}^{\lfloor\frac{n+1}{3}\rfloor}\binom{n+d+1}{4d}
\end{equation}
and
\begin{equation}\label{c2e2}
a(2n+1)=\sum_{d=0}^{\lfloor\frac{n}{3}\rfloor}\binom{n+d+2}{4d+2},
\end{equation}
where $a(m)$ is given by \eqref{p2e1}.
\end{corollary}
Note that both sides of \eqref{c2e1} and \eqref{c2e2} are seen to count the parity-alternating compositions of length $2n$ and $2n+1$, respectively, the right-hand side by the number of parts (once one applies the identity $\binom{a}{b}=\binom{a-2}{b}+2\binom{a-2}{b-1}+\binom{a-2}{b-2}$, which has an easy combinatorial explanation, to the binomial coefficient).  Using \eqref{p2e1}, the binomial sums in \eqref{c2e1} and \eqref{c2e2} can be shown to satisfy fourth order recurrences; see \cite{CE} for other examples of recurrent binomial sums.

\section{Asymptotics}
\label{s:asymptotics}

We recall from \eqref{c1e1} that $F_m(x,y,q)$ is a rational function.
Specializing variables we obtain $F_m(x,1,0)=\sum_{n\geq0}a_nx^n$, which we have seen in the
previous section. The exact formulae given there for the coefficient
$a_n$ are complemented here by asymptotic results. These are analogous
to known results for Smirnov words and Carlitz compositions.

Note that $F_m(x,1,0) = 1/H_m(x)$, where $H_m(x) = 1 - \sum_{a=1}^m
\frac{x^a}{1 + x^a - x^m}$. Each $1+x^a-x^m$ is analytic and so its
modulus over each closed disk centered at $0$ is maximized on the
boundary circle. It can be shown that when $|x|$ is fixed,
$|1+x^a-x^m|$ is maximized when $x^a - x^m$ is positive real, and
minimized when $x^a - x^m$ is negative real. Furthermore, the maximum
over $a$ of this maximum value occurs when $a=1$, and similarly for the
minimum.

By Pringsheim's theorem, there is a minimal singularity of $F_m$ on the
positive real axis, and this is precisely the smallest real zero $\rho_m$
of $H_m$. Furthermore, because $F_m$ is not periodic, this singularity is
the unique one of that modulus. Thus $F_m$ is analytic in the open disk
centered at $0$ with radius $\rho_m$. Note that $\rho_m \geq 1/2$ because the
exponential growth rate of unrestricted compositions is $2$, and so our
restricted class of  compositions must grow no faster. However $\rho_m
\leq 1$ because the sum defining $H_m$ has value $m$ when $x = 1$.
Since $\rho_m$ is the smallest positive real solution of
$$
\sum_{a=1}^m \frac{\rho^a}{1 + \rho^a - \rho^m} = 1,
$$
it follows that $\rho_m$ is an algebraic number of degree at most $m^2$.
Note that the sum defining $H_m$ shows that $H_m(x) > H_{m+1}(x)$ for
all $m$ and all $0<x<1$. Thus in fact $\rho_m \leq \rho_2 < 0.68$ for all $m$.

The rest of the proof should proceed according to a familiar outline:
apply Rouch\'{e}'s theorem to locate the dominant singularity of
$F_m(x,1,0)$, by approximating $H_m$ with a simpler function having a
unique zero inside an appropriately chosen disk of radius $c$, where
$\rho_m < c < 1$; derive asymptotics for the coefficients $a_n$ via
standard singularity analysis. This technique has been used in several
similar problems, for example for Carlitz compositions.

There are some difficulties with this approach in our case. If we attack
$F_m$ directly, we must derive a result for all $m$. Since $F_m$ is
rational with numerator and denominator of degree at most $m^2$, for
fixed $m$, we could consider using numerical root-finding methods, but
for arbitrary symbolic $m$ this will not work. It is intuitively clear
that for sufficiently large $m$, $F_m$ should be close to $F_\infty$ and
so by using Rouch\'{e}'s theorem, we could reduce to the Carlitz case.

However, even the Carlitz case is not as easy as claimed in the
literature, and we found several unconvincing published arguments.
 Some authors simply assert that
$F_\infty$ has a single root, based on a graph of the function on a
given circle. This can be made into a proof, by approximately evaluating
$F_\infty$ at sufficiently many points and using an upper bound on the
best Lipschitz constant for the function, but this is somewhat
unpleasant. We do not know a way of avoiding this problem --- the minimum modulus of
a function on a circle in the complex plane must be computed somehow. We use an
approach similar to that taken in \cite{GoHi2002}.

The obvious approximating function to use is an initial segment with $k$
terms of the partial sum defining $F_m$. However, it seems easier to use the initial segment
of the sum defining $F_\infty$, which we denote by $S_k$. We will take $k=7$ and $c=0.7$ and denote $S_7$ by $h$. By
using the Jenkins-Traub algorithm as implemented in the Sage command
\texttt{maxima.allroots()}, we see that all roots of $h$, except the real positive
root (approximately $0.572$) have real or imaginary part with modulus more
than $0.7$, so they certainly lie outside the circle $C$ given by $|x| = c$.

To apply Rouch\'{e}'s theorem, we need an upper bound for $|H_m - h|$ on
$C$ which is less than the lower bound for $|h|$ on $C$. We first claim
that the lower bound for $|h|$ on $C$ is at least $0.43$. This can be
proved by evaluation at sufficiently many points of $C$. Since $|h'(z)|$ is bounded by
$100$ on $C$, $N:=1000$ points certainly suffice. This is because the minimum
of $|h|$ at points of the form $0.7\exp(2\pi ij/N)$, for $0\leq j
\leq N$, is more than $0.51$ (computed using Sage), and the distance
between two such points is at most $8\times 10^{-4}$, by Taylor
approximation. In fact, it seems that the minimum indeed occurs at
$x=0.7$, but this is not obvious to us.

We now compute an upper bound for $|H_m - h|$. To this end, we compute, when $m\geq 7$,
\begin{align*}
\left| H_m(x) - h(x)\right| & = \sum_{a=8}^m \frac{x^a}{1 - x^a + x^m} +
\left(\sum_{a=1}^7 \frac{x^a}{1 - x^a + x^m}  - h(x) \right)\\
& \leq \sum_{a=8}^\infty \frac{c^a}{1 - c^8} + \left(\sum_{a=1}^7
\frac{x^a}{1 - x^a + x^m}  - h(x) \right).
\end{align*}
The sum $\sum_{a=8}^\infty \frac{c^a}{1 - c^8}$ has  value less than
$0.204$ when $c=0.7$. The second sum is smaller than $0.2$ which can be
verified by a similar argument to the above, by evaluating at
sufficiently many points.

We still need to deal with the cases $m < 7$ and these can be done directly via inspection after computing all roots numerically as above.

The above arguments show that $\rho_m$ is a simple zero of $H_m$ and
hence a simple pole of the rational function $F_m(x,1,0)$. The
asymptotics now follow in the standard manner by a residue computation,
and we obtain
$$
a_n \sim \rho_m^{-n} \frac{1}{-\rho_m H'(\rho_m)}.
$$

For example, $F_2(x,1,0) = (1+x+x^2)/((1-x^2)^2 - x^3)$ has a minimal
singularity at $\rho_2 \approx 0.6710436067037893$, which yields the following result.

\begin{theorem}\label{asymth}  We have
$$a(n)  \sim (0.6436) 1.4902^n$$
for large $n$, where $a(n)$ is given by \eqref{p2e1}.
\end{theorem}

For example, when $n=20$, the relative error in this approximation is already less than
$0.2\%$. The exponential rate $1/\rho_m$ approaches the rate for Carlitz compositions,
namely $1.750\cdots$, as $m\to\infty$.

We can in fact derive asymptotics in the multivariate case. For each $m$, it is possible in principle to compute asymptotics in a given direction by analysis of $F_m(x,y,q)$, for example, using the techniques of Pemantle and Wilson \cite{PeWi2002}. To do this for arbitrarily large $m$ is computationally challenging, and so in order to limit the length of this article, we give a sketch only for $m=2$, and refer the reader to the above reference or the more recent book \cite{PeWi2013}.
In this case we have
\begin{align*}
F_2(x,y,q) & = \left( 1- \frac{xy}{1 - x^2 - xy(q-1)} + \frac{x^2y}{1 - x^2 - x^2y(q-1)}\right)^{-1}\\
& = \frac{(1 - x^2 - xy(q-1))(1 - x^2 - x^2y(q-1))}{1-2x^2-qxy+x^4-qx^2y-x^3y^2+qx^3y+qx^4y+q^2x^3y^2}.
\end{align*}

By standard algorithms, for example as implemented in Sage's
\texttt{solve} command, one can check that the partial derivatives $H_x,
H_y, H_q$ never vanish simultaneously, so that the variety defined by
$H_m$ is smooth everywhere. The critical point equations are readily solved by the same
method. For example, for the special case when $n=2d=4t$, where $t$ denotes the number of congruence successions, we obtain (using the Sage package \texttt{amgf} \cite{amgf}) the first order asymptotic
$$
(0.379867842273) (15.8273658508862)^t/(\pi t),
$$
which has relative error just over $1\%$ when $n=32$ (the number of such compositions being
54865800). Bivariate asymptotics when $q=0$, or when $y=1$, could be derived similarly. The smoothness of the variety defined by $H_m$ leads quickly to Gaussian limit laws in a standard way as described in \cite{PeWi2013}, and we leave the reader to explore this further.


\end{document}